\documentclass[11pt,english]{amsart}
\usepackage{ae,aecompl}
\usepackage[T1]{fontenc}
\usepackage[latin9]{inputenc}
\usepackage{amssymb}

\makeatletter
\numberwithin{equation}{section} 
\numberwithin{figure}{section} 
  \theoremstyle{plain}
  \newtheorem{thm}{Theorem}[section]
  \theoremstyle{definition}
  \newtheorem{defn}[thm]{Definition}
  \theoremstyle{plain}
  \newtheorem{lem}[thm]{Lemma}
  \theoremstyle{remark}
  \newtheorem{rem}[thm]{Remark}
  \theoremstyle{plain}
  \newtheorem{prop}[thm]{Proposition}
  \theoremstyle{remark}
  \newtheorem*{rem*}{Remark}
  \theoremstyle{plain}
  \newtheorem{cor}[thm]{Corollary}
  \theoremstyle{remark}
  \newtheorem*{acknowledgement*}{Acknowledgement}

\usepackage{slashed}

\usepackage{babel}
\makeatother

\begin{document}

\title{Poisson suspensions and infinite ergodic theory}

\author{Emmanuel Roy}

\curraddr{Laboratoire Analyse Géométrie et Applications, UMR 7539, Université
Paris 13, 99 avenue J.B. Clément, F-93430 Villetaneuse, France}

\email{roy@math.univ-paris13.fr}

\begin{abstract}
We investigate ergodic theory of Poisson suspensions. In the process,
we establish close connections between finite and infinite measure
preserving ergodic theory. Poisson suspensions thus provide a new
approach to infinite measure ergodic theory. Fields investigated here
are mixing properties, spectral theory, joinings. We also compare
Poisson suspensions to the apparently similar looking Gaussian dynamical
systems.
\end{abstract}


\maketitle

\section{Introduction}

Poisson measures are classical constructions from probability theory.
They are the discrete counterpart of Brownian motion and they together
give way to an enormous proportion of probabilistic investigations.
In measure preserving ergodic theory, it is a functorial way to associate
two dynamical systems: one associated to a probability measure (the
Poisson measure) and one associated to a $\sigma$-finite measure
(the base of the Poisson measure). The measure preserving transformation
$T$ that sends a point to another on the base becomes a probability
preserving transformation $T_{*}$ on points configurations by sending
a configuration to another where each point is sent to its image according
to $T$. A Poisson suspension is a Poisson measure with this kind
of transformation.

After some generalities on Poisson suspensions, we recall the basics
of infinite measure preserving ergodic theory that are needed to study
Poisson suspensions. Theorem \ref{thm:ergopropPoissusp} is the interpretation
of mixing properties of the suspension in terms of  asymptotic behaviours
of the base. Part of these results was known for a long time but we
believe it is worth giving the complete picture.

We then introduce Poissonian joinings, whose structure is completely
clarified in Theorem \ref{thm:structureIDjoinings} and based on strong
connections with joinings at the bases level. These joinings (under
different names) were introduced in parallel in \cite{Lem05ELF} and
in the PhD thesis of the author, with a different, totally probabilistic,
but equivalent definition. In this paper we recall both definitions
and extend the characterization of these joinings. A specific notion
of disjointness (called ID-disjointness), weaker than the usual notion
can then be introduced and provides a probabilistic analogous of the
notion of strong disjointness introduced by Aaronson for infinite
measure preserving transformations. In particular (Theorem \ref{thm:disfortdisID})
we prove that two ergodic systems are strongly disjoint if and only
if their Poisson suspensions are ID-disjoint. Since disjointness implies
ID-disjointness, this result gives a useful {}``finite measure''
criterion of strong disjointness. Surprisingly, ID-disjointness mixed
with the theory of Gaussian dynamical systems allows us to derive
a new result on spectral theory of infinite ergodic measure preserving
transformation: we prove that a whole family of continuous measures
(among them, measures supported on Kronecker sets) cannot be realised
as spectral measures.

In the final chapter, as an illustration of what we have developed
so far, we focus on Poisson suspensions with simple spectrum (in fact,
the results in this chapter are a bit more general) where many nice
structural features appear very clearly, making the functorial relationship
between the base and the suspension even closer. Throughout this work,
we will see many results that are similar to the Gaussian case, proofs
being often different. However, this final chapter shows some drastic
differences between Poisson suspensions and Gaussian systems.

The tools developed in this paper will find further applications on
papers in preparation, one specifically devoted to Poissonian joinings
(\cite{ParRoy07Poisjoin} in collaboration with François Parreau),
the other on the entropy of infinite measure systems (\cite{Jan07entropy},
in collaboration with Elise Janvresse, Tom Meyerovitch and Thierry
de la Rue).

\section{Poisson suspensions}

\subsection{Poisson measure}

Let $\left(X,\mathcal{A},\mu\right)$ be a $\sigma$-finite Lebesgue
space with a continuous infinite measure $\mu$ and consider $\left(X^{*},\mathcal{A}^{*}\right)$
the space of measures on $\left(X,\mathcal{A}\right)$ where $\mathcal{A}^{*}$
is the $\sigma$-algebra generated by the functions $N\left(A\right):\nu\mapsto\nu\left(A\right)$,
$A\in\mathcal{A}$, $\nu\in X^{*}$.

There exists an isomorphism $\varphi$ between $\left(X,\mathcal{A},\mu\right)$
and $\left(\mathbb{R},\mathcal{B},\lambda\right)$ where $\lambda$
is the Lebesgue measure. Let $\widetilde{X^{*}}\subset X^{*}$ be
the measurable set defined by the relations $\nu\in\widetilde{X^{*}}\Longleftrightarrow\forall n\in\mathbb{Z},\nu\left(\varphi^{-1}\left[n,n+1\right]\right)\in\mathbb{N}$.
On $\left(\widetilde{X^{*}},\widetilde{\mathcal{A}^{*}}\right)$,
we can easily form the probability measure $\mu^{*}$ such that for
all finite family of disjoint finite $\mu$-measure sets $A_{1},\dots,A_{k}$,
the $N\left(A_{i}\right)$ are independent and distributed as Poisson
random variables of parameter $\mu\left(A_{i}\right)$. We can extend
$\mu^{*}$ to $\left(X^{*},\mathcal{A}^{*}\right)$ by setting $\mu^{*}\left(X^{*}\setminus\widetilde{X^{*}}\right)=0$.
If we consider that $\mathcal{A}^{*}$ is  completed with respect
to $\mu^{*}$, $\left(X^{*},\mathcal{A}^{*},\mu^{*}\right)$ is a
Lebesgue space and called the Poisson measure over the base $\left(X,\mathcal{A},\mu\right)$.
This construction is independent of the choice of the isomorphism
$\varphi$.

\subsection{\label{sub:Poissonian-map-and}Poissonian map, Poissonian factor,
Fock space structure}

If $\Psi$ is a measure preserving map from $\left(X_{1},\mathcal{A}_{1},\mu_{1}\right)$
to $\left(X_{2},\mathcal{A}_{2},\mu_{2}\right)$, call the map $\Psi_{*}:\nu\mapsto\nu\circ\Psi^{-1}$
a \emph{Poissonian map}. $\Psi_{*}$ is a measure preserving map from
$\left(X_{1}^{*},\mathcal{A}_{1}^{*},\mu_{1}^{*}\right)$ to $\left(X_{2}^{*},\mathcal{A}_{2}^{*},\mu_{2}^{*}\right)$.

Let $T$ be an automorphism of $\left(X,\mathcal{A},\mu\right)$,
the Poissonian map $T_{*}$ is an automorphism of $\left(X^{*},\mathcal{A}^{*},\mu^{*}\right)$
called \emph{Poissonian automorphism}. The dynamical system $\left(X^{*},\mathcal{A}^{*},\mu^{*},T_{*}\right)$
is the \emph{Poisson suspension} over the base $\left(X,\mathcal{A},\mu,T\right)$.

Consider $\mathcal{F}$ a factor of $\mathcal{A}$ (where $\mu$ restricted
to $\mathcal{A}$ is not necessarily $\sigma$-finite), then $\mathcal{F}^{*}:=\sigma\left\{ N\left(F\right),\; F\in\mathcal{F}\right\} $
is a factor of $\mathcal{A}^{*}$ called the Poissonian factor associated
to $\mathcal{F}$. Two factors of $\mathcal{A}$ which don't differ
on positive finite measure sets have the same Poissonian factor.

It is well known that $L^{2}\left(\mu^{*}\right)$ is the Fock space
over $L^{2}\left(\mu\right)$, i.e. $L^{2}\left(\mu^{*}\right)\simeq\mathbb{C}\oplus L^{2}\left(\mu\right)\oplus L^{2}\left(\mu\right)^{\otimes2}\oplus\cdots\oplus L^{2}\left(\mu\right)^{\otimes n}\oplus\cdots$,
the isomorphism being given through multiple stochastic integrals.
Considered as a subspace of $L^{2}\left(\mu^{*}\right)$, $L^{2}\left(\mu\right)^{\otimes n}$
is called the\emph{ $n$-th chaos}.

An operator $\Psi$ with norm less than $1$ on $L^{2}\left(\mu\right)$
gives way to an operator $\widetilde{\Psi}$ on $L^{2}\left(\mu^{*}\right)$
which acts as $\Psi^{\otimes n}$ on $L^{2}\left(\mu\right)^{\otimes n}$,
we call $\widetilde{\Psi}$ the \emph{exponential} of $\Psi$. In
particular, $U_{T_{*}}$ is the exponential of $U_{T}$, therefore,
if $\sigma$ is the maximal spectral type of $T$, $\sum_{k\geq1}\frac{1}{k!}\sigma^{*k}$
is the reduced maximal spectral type of $T_{*}$.

If $\mathcal{F}$ is a $\sigma$-finite factor and $E_{\mathcal{F}}$
the corresponding conditional expectation, it is easy to see that
$E_{\mathcal{F}^{*}}$, the conditional expectation corresponding
to $\mathcal{F}^{*}$ is the\emph{ }exponential\emph{ }of $E_{\mathcal{F}}$.

\section{infinite ergodic theory and mixing properties of suspensions}

\subsection{Categories of dynamical systems}

Consider a system $\left(X,\mathcal{A},\mu,T\right)$ where $\left(X,\mathcal{A},\mu\right)$
is a Lebesgue space and $T$ a measure preserving automorphism.

Most of the following terminology is taken from \cite{KrenSuch69MixInf}.

\begin{defn}
$T$ is said to be:

- \textbf{dissipative}, if there exists a set $W$ such that $T^{n}W\cap T^{m}W=\emptyset$
($W$ is called a \emph{wandering set}) if $n\neq m$ and $X=\bigcup_{n\in\mathbb{Z}}T^{n}W$
mod. $\mu$.

- \textbf{conservative}, if there is no wandering set.

- \textbf{remotely infinite} if there exists a $\sigma$-finite sub-$\sigma$-field
$\mathcal{G}\subset\mathcal{A}$ such that, mod. $\mu$:

\[
T^{-1}\mathcal{G}\subseteq\mathcal{G}\]

\[
\vee_{n\in\mathbb{Z}}T^{-n}\mathcal{G}=\mathcal{A}\]
and

\[
\mathcal{A}_{f}\cap_{n\in\mathbb{Z}}T^{-n}\mathcal{G}=\emptyset\]
 where $\mathcal{A}_{f}$ are the elements of $\mathcal{A}$ of finite
positive $\mu$-measure.

- of \textbf{zero type}, if for all set $A$ of finite measure,

\[
\textrm{lim}_{n\to\infty}\mu\left(A\cap T^{n}A\right)=0\]

- of \textbf{type $\mathbf{II}_{\infty}$} if it is conservative and
if there is no $T$-invariant probability measure $\tilde{\mu}$ such
that $\tilde{\mu}\ll\mu$.

- \textbf{rigid} if there exists a sequence $n_{k}\uparrow\infty$
such that $\textrm{lim}_{k\to\infty}\mu\left(A\bigtriangleup T^{n_{k}}A\right)=0$.
\end{defn}
The notion of rigidity allows to introduce a new notion for \textbf{type
$\mathbf{II}_{\infty}$} systems, parallel to mild mixing for probability
measure preserving systems.

\begin{defn}
Consider a dynamical system $\left(X,\mathcal{A},\mu,T\right)$.

$T$ is said to be \textbf{rigidity free} if for every sequence $n_{k}\uparrow\infty$
and set $A$ of finite measure, $\underline{\textrm{lim}}_{k\to\infty}\mu\left(A\bigtriangleup T^{n_{k}}A\right)>0$.
\end{defn}

\subsection{Mixing properties of Poisson suspensions}

We will make the connections between the above definitions and the
mixing properties of the corresponding Poisson suspensions. First
we need this simple continuity lemma:

\begin{lem}
\label{lem:continuity}Let $\left\{ \mathcal{A}_{n}\right\} _{n\in\mathbb{N}}$
be a collection of $\sigma$-finite sub-$\sigma$-fields of $\mathcal{A}$
such that $\mathcal{A}_{n}\subset\mathcal{A}_{m}$ if $n\geq m$.
We have:

\[
\left(\cap_{n\in\mathbb{N}}\mathcal{A}_{n}\right)^{*}=\cap_{n\in\mathbb{N}}\mathcal{A}_{n}^{*}\]
and, with $\left\{ \mathcal{A}_{n}\right\} _{n\in\mathbb{N}}$ a collection
of $\sigma$-finite sub-$\sigma$-fields of $\mathcal{A}$ such that
$\mathcal{A}_{n}\subset\mathcal{A}_{m}$ if $n\leq m$

\[
\left(\vee_{n\in\mathbb{N}}\mathcal{A}_{n}\right)^{*}=\vee_{n\in\mathbb{N}}\mathcal{A}_{n}^{*}\]

\end{lem}
\begin{proof}
First note that $E_{\mathcal{A}_{n}}$ tends to $E_{\cap_{n\in\mathbb{N}}\mathcal{A}_{n}}$
weakly on $L^{2}\left(\mu\right)$ and $E_{\mathcal{A}_{n}^{*}}$
tends to $E_{\cap_{n\in\mathbb{N}}\mathcal{A}_{n}^{*}}$ weakly on
$L^{2}\left(\mu^{*}\right)$. Since $E_{\mathcal{A}_{n}^{*}}$ acts
as $E_{\mathcal{A}_{n}}^{\otimes n}$ on $L^{2}\left(\mu\right)^{\otimes n}$
which tends to $E_{\cap_{n\in\mathbb{N}}\mathcal{A}_{n}}^{\otimes n}$,
we deduce that $E_{\cap_{n\in\mathbb{N}}\mathcal{A}_{n}^{*}}$ coincide
with $E_{\cap_{n\in\mathbb{N}}\mathcal{A}_{n}}^{\otimes n}$ on $L^{2}\left(\mu\right)^{\otimes n}$.
That is $E_{\cap_{n\in\mathbb{N}}\mathcal{A}_{n}^{*}}=E_{\left(\cap_{n\in\mathbb{N}}\mathcal{A}_{n}\right)^{*}}$
which proves the first point. The proof of the second assertion is
similar.
\end{proof}
\begin{thm}
\label{thm:ergopropPoissusp}

$T$ is of type $\mathbf{II}_{\infty}$ on its conservative part iff
$T_{*}$ is ergodic (and then weakly mixing).

$T$ is rigid (for the sequence $n_{k}$) iff $T_{*}$ is rigid (for
the sequence $n_{k}$).

$T$ is rigidity free iff $T_{*}$ is mildly mixing.

$T$ is of zero type iff $T_{*}$ is mixing (and then mixing of any
order).

If $T$ is remotely infinite, then $T_{*}$ is $\mathrm{K}$.

If $T$ is dissipative, then $T_{*}$ is Bernoulli.
\end{thm}
\begin{proof}
The first and fourth points are apparently due to Marchat in his PhD
dissertation as pointed out by Grabinski in \cite{Grab84Poiss}. We
have given our proof of these facts in \cite{Roy06IDstat}. The sixth
point is folklore.

If $T$ is rigid for the sequence $n_{k}$, any vector of the first
chaos of $T_{*}$ is a rigid vector for the sequence $n_{k}$. But
since the smallest $\sigma$-algebra containing the first chaos is
the whole $\sigma$-algebra $\mathcal{A}^{*}$, $T_{*}$ is rigid
for the sequence $n_{k}$. The converse is obvious.

If $T$ is rigidity free, it means that its maximal spectral type
$\sigma$ doesn't put mass on weak Dirichlet sets. It is easily checked
that this property is also shared by its successive convolution powers,
and then by the measure $\sum_{k\geq1}\frac{1}{k!}\sigma^{*k}$. This
implies that $T_{*}$ is mildly mixing. Once again, the converse is
obvious.

If $T$ is remotely infinite, by taking $\mathcal{G}$ as in the definition,
$\mathcal{G}^{*}$ satisfies: 

\[
T_{*}^{-1}\mathcal{G}^{*}\subseteq\mathcal{G}^{*}\]

\[
\vee_{n\in\mathbb{Z}}T_{*}^{-n}\mathcal{G}^{*}=\mathcal{A}^{*}\]

and\[
\cap_{n\in\mathbb{Z}}T_{*}^{-n}\mathcal{G}^{*}=\left\{ \emptyset,X^{*}\right\} \]

thanks to Lemma \ref{lem:continuity}. Thus $T_{*}$ is $\mathrm{K}$.
\end{proof}

\section{Poissonian joinings}

We present here a natural family of joinings between Poisson suspensions.These
joinings were introduced in \cite{Lem05ELF} and in parallel, with
a different point of view in \cite{Roy05these} under the name \emph{ID
joinings}.

\subsubsection{Infinite divisibility of a Poisson suspension}

Loosely speaking, a probability measure is infinitely divisible (abr.
ID) if it is, for all positive integer $n$, the $n$-th convolution
power of certain probability measure (see \cite{Sato99LevPro} for
infinite divisibility on $\mathbb{R}^{d}$).

Consider two bases $\left(X,\mathcal{A},\mu\right)$ and $\left(Y,\mathcal{B},\nu\right)$.
The sum of two elements in $\left(X^{*},\mathcal{A}^{*}\right)$ and
$\left(Y^{*},\mathcal{B}^{*}\right)$ is defined as the usual sum
of two measures. The convolution operation $*$ is thus well defined
for distributions on $\left(X^{*},\mathcal{A}^{*}\right)$ and $\left(Y^{*},\mathcal{B}^{*}\right)$,
and so is infinite divisibility. The well known formula $\mu^{*}=\left(\left(\frac{1}{n}\mu\right)^{*}\right)^{*n}$for
all $k\geq1$ makes the intrinsic ID nature of a Poisson measure precise.

Infinite divisibility is also well defined on $\left(X^{*}\times Y^{*},\mathcal{A}^{*}\otimes\mathcal{B}^{*}\right)$,
the convolution operation coming from the sum coordinatewise.

\begin{defn}
A \emph{Poissonian joining} of $\left(X^{*},\mathcal{A}^{*},\mu^{*},T_{*}\right)$
and $\left(Y^{*},\mathcal{B}^{*},\nu^{*},S_{*}\right)$ is any joining
$\left(X^{*}\times Y^{*},\mathcal{A}^{*}\otimes\mathcal{B}^{*},\rho,T_{*}\times S_{*}\right)$
such that $\rho$ is ID.
\end{defn}

\subsubsection{KLM measure}

To derive results on the structure of Poissonian joinings of Poisson
suspensions, we need to consider the notion of KLM measure (see \cite{Matthes81IDPP}
and \cite{DaleyVere-Jones88Point}) of an ID point process which plays
exactly the same role as does the Lévy measure for ID processes (see
\cite{Maruyama70IDproc}). KLM measures are measures on $\left(X^{*},\mathcal{A}^{*}\right)$
and their existence is ensured if $\left(X,\mathcal{A}\right)$ is
a complete separable metric space. As it can be checked from the definition
in \cite{Matthes81IDPP}, KLM measures can be extended to the product
space $\left(X^{*}\times X^{*},\mathcal{A}^{*}\otimes\mathcal{A}^{*}\right)$
corresponding to bivariate ID point processes.

Here $\left(X,\mathcal{A}\right)$ stands for the real line or the
disjoint union of two copies of it, endowed with their Borel sets.

If $\mathbb{P}$ is the distribution of an ID point process on $\left(X,\mathcal{A}\right)$,
there exists a unique $\sigma$-finite measure $Q$ on $\left(X^{*},\mathcal{A}^{*}\right)$
such that $Q\left(\mu_{0}\right)=0$ ($\mu_{0}$ denotes the null
measure) which characterizes completely the distribution of an ID
point process through the formula:

${\displaystyle \int_{X^{*}}}\exp\left(-{\displaystyle \int_{X}}f\left(x\right)\nu\left(dx\right)\right)\mathbb{P}\left(d\nu\right)$

\[
=\exp\left[\int_{X^{*}}\left(\exp-{\displaystyle \int_{X}}f\left(x\right)\nu\left(dx\right)\right)-1Q\left(d\nu\right)\right]\]

for any nonnegative function with compact support $f$.

If $m$ is a Radon measure on $\left(X,\mathcal{A}\right)$ and $m^{*}$
the distribution of the Poisson measure over $\left(X,\mathcal{A},m\right)$,
it is easy to see, thanks to the classical formula,

\[
\int_{X^{*}}\exp\left(-{\displaystyle \int_{X}}f\left(x\right)\nu\left(dx\right)\right)m^{*}\left(d\nu\right)=\exp\left[\int_{X}\left(e^{-f\left(x\right)}-1\right)m\left(dx\right)\right]\]

that the KLM measure $Q_{m}$ of $m^{*}$ is the image measure of
$m$ by the application $\Delta:x\to\delta_{x}$ from $\left(X,\mathcal{A}\right)$
to $\left(X_{0}^{*},\mathcal{A}_{0}^{*}\right)$ where $X_{0}^{*}\subset X^{*}$
is the set of Dirac masses. Then we can describe the KLM measure of
a Poisson measure as the measure consisting of one dirac mass {}``randomized''
according to the measure on the base.

\subsubsection{Structure of Poissonian joinings}

We will deduce the structure of a Poissonian joining of distribution
$\mathfrak{P}$ of Poisson suspensions built over $\left(X,\mathcal{A},\mu,T\right)$,
$\left(Y,\mathcal{B},\nu,S\right)$ where $\left(X,\mathcal{A},\mu\right)=\left(Y,\mathcal{B},\nu\right)=\left(\mathbb{R},\mathcal{B},\lambda\right)$.

Let $\mathfrak{P}$ be the distribution of a Poissonian joining between
$\left(X,\mathcal{A},\mu,T\right)$ and $\left(Y,\mathcal{B},\nu,S\right)$.
Call $Q_{\mathfrak{P}}$ the KLM measure of $\mathfrak{P}$ and decompose
it along the $T_{*}\times S_{*}$-invariant subsets $A=X^{*}\times\left\{ \mu_{0}\right\} $,
$B=\left\{ \mu_{0}\right\} \times Y^{*}$ and $C=X^{*}\times Y^{*}\setminus\left(A\cup B\right)$,
thus $Q_{\mathfrak{P}}=Q_{1}\otimes\delta_{\mu_{0}}+\delta_{\mu_{0}}\otimes Q_{2}+\widetilde{Q}$
where the three members of the sum are the restriction to $A$, $B$
and $C$ of $Q_{\mathfrak{P}}$. $Q_{1}$ and $Q_{2}$ are the KLM
measures of point processes on $X^{*}$ and $Y^{*}$ and satisfy $Q_{\mu}=Q_{1}+\widetilde{Q}\circ\pi_{X}^{-1}$
and $Q_{\nu}=Q_{2}+\widetilde{Q}\circ\pi_{Y}^{-1}$.

Since $Q_{1}\ll Q_{\mu}$, $Q_{2}\ll Q_{\nu}$, $\widetilde{Q}\circ\pi_{X}^{-1}\ll Q_{\mu}$
and $\widetilde{Q}\circ\pi_{2}^{-1}\ll Q_{\nu}$ , all of these measures
are supported on the set $\left\{ \delta_{x},\: x\in\mathbb{R}\right\} $
and $\widetilde{Q}$ is supported on the set $\left\{ \delta_{x},\: x\in\mathbb{R}^{2}\right\} $.
Denote by $\mu^{\prime}=Q_{1}\circ\Delta$, $\nu^{\prime}=Q_{2}\circ\Delta$,
$\tilde{\mu}=\left(\widetilde{Q}\circ\pi_{X}^{-1}\right)\circ\Delta$,
$\tilde{\nu}=\left(\widetilde{Q}\circ\pi_{Y}^{-1}\right)\circ\Delta$
and $\gamma=\widetilde{Q}\circ\left(\Delta\times\Delta\right)$.

If $\widetilde{\gamma^{*}}$ denotes the image of $\gamma^{*}$ by
the mapping $\pi_{X}^{*}\times\pi_{Y}^{*}$, we have:

\[
\mathfrak{P}=\left(\mu^{\prime*}\otimes\delta_{\mu_{0}}\right)*\widetilde{\gamma^{*}}*\left(\delta_{\mu_{0}}\otimes\nu^{\prime*}\right)\]

We have proved the following structure theorem:

\begin{thm}
\label{thm:structureIDjoinings}Let $\left(X^{*},\mathcal{A}^{*},\mu^{*},T_{*}\right)$
and $\left(Y^{*},\mathcal{B}^{*},\nu^{*},S_{*}\right)$ be two Poisson
suspensions. Let $\left(X^{*}\times Y^{*},\mathcal{A}^{*}\otimes\mathcal{B}^{*},\mathfrak{P},T_{*}\times S_{*}\right)$
be a Poissonian joining between them (i.e. $\mathfrak{P}$ is ID).
There exists $T$-invariant measures $\mu^{\prime}$ and $\tilde{\mu}$,
$S$-invariant measures $\mu^{\prime}$ and $\tilde{\mu}$ such that
$\mu=\mu^{\prime}+\tilde{\mu}$, $\nu=\nu^{\prime}+\tilde{\nu}$ and
$\gamma$ is a $T_{1}\times T_{2}$-invariant joining of $\tilde{\mu}$
and $\tilde{\nu}$. If $\tilde{\gamma^{*}}$ denotes the image of
$\gamma^{*}$ by the mapping $\pi_{X}^{*}\times\pi_{Y}^{*}$, we have:

\[
\mathfrak{P}=\left(\mu^{\prime*}\otimes\delta_{\mu_{0}}\right)*\widetilde{\gamma^{*}}*\left(\delta_{\mu_{0}}\otimes\nu^{\prime*}\right)\]

\end{thm}
\begin{rem}
We mention that Poissonian joinings of higher order can be defined
in an obvious way. It is then easily checked that these joinings satisfy
the PID property, i.e. any pairwise independent joining is in fact
independent (two vectors $\left(X_{1},X_{2},\dots,X_{n}\right)$ and
$\left(Y_{1},Y_{2},\dots,Y_{k}\right)$ such that their joint distribution
is ID are independent if and only if for any choice $\left(i,j\right)$,
$X_{i}$ and $Y_{j}$ are independent).
\end{rem}

\subsubsection{Poissonian joinings and Markov operators}

We present the point of view taken in \cite{Lem05ELF}.

A sub-Markov operator between $L^{2}\left(\mu\right)$ and $L^{2}\left(\nu\right)$
is a non-negative operator $\Phi$ such that $\Phi1\leq1$ and $\Phi^{*}1\leq1$.

A Poissonian joining in the sense of \cite{Lem05ELF} is a joining
such that the corresponding Markov operator acting between $L^{2}\left(\mu^{*}\right)$
and $L^{2}\left(\nu^{*}\right)$ acts as the exponential of a sub-Markov
operator $\Phi$ between $L^{2}\left(\mu\right)$ and $L^{2}\left(\nu\right)$
intertwinning $U_{T}$ and $U_{S}$. The author then proves that these
joinings have the structure described in Theorem \ref{thm:structureIDjoinings}
thus defining the same family of joinings.

We will show that the hypothesis that $\Phi$ is a sub-Markov operator
is automatically satisfied if $\Phi^{*}$ is a Markov operator. But
let us prove first a slightly more general result which will be needed
under this form in the last chapter:

\begin{prop}
\label{pro:Markov}Let $\left(X^{*},\mathcal{A}^{*},\mu^{*}\right)$
and $\left(Y^{*},\mathcal{B}^{*},\nu^{*}\right)$ be two Poisson measures
and $\Psi$, a Markov operator from $L^{2}\left(\mu^{*}\right)$ to
$L^{2}\left(\nu^{*}\right)$ which sends $\mathfrak{C}_{X}$, the
first chaos of $L^{2}\left(\mu^{*}\right)$ to $\mathfrak{C}_{Y}$,
the first chaos of $L^{2}\left(\nu^{*}\right)$. Then the restriction
of $\Psi$ to $\mathfrak{C}_{X}$ induces a sub-Markov operator $\Theta$
from $L^{2}\left(\mu\right)$ to $L^{2}\left(\nu\right)$.
\end{prop}
\begin{proof}
$\Theta$ is defined through the identification between $\mathfrak{C}_{X}$
and $L^{2}\left(\mu\right)$, and $\mathfrak{C}_{Y}$ and $L^{2}\left(\nu\right)$.
We have to show it is a positive operator. It will be a consequence
of the following fact: non-negative random variables belonging to
$\mathbb{C}\oplus\mathfrak{C}_{Y}$ are of the form $\int_{Y}fdN+c$
where $f\in L^{1}\left(\nu\right)\cap L^{2}\left(\nu\right)$, $f\geq0$
and $c\geq0$.

Let $N_{X}$ and $N_{Y}$ denote the random measures with distribution
$\mu^{*}$ and $\nu^{*}$. $M_{X}$ and $M_{Y}$ denotes the extended
centered random measures, i.e. $M_{X}\left(g\right):=\lim_{n\to\infty}\left(\int_{Y}g_{n}dN_{Y}-\int_{Y}g_{n}d\nu\right)$
in $L^{2}\left(\nu^{*}\right)$, $g_{n}\in L^{1}\left(\nu\right)\cap L^{2}\left(\nu\right)$,
$g\in L^{2}\left(\nu\right)$, $\lim_{n\to\infty}g_{n}=g$ in $L^{2}\left(\nu\right)$.

First, we recall that the Lévy measure $\rho$ of a positive infinitely
divisible integrable random variable $U$ satisfies $\rho\left(\mathbb{R}_{-}\right)=0$,
$\int_{\mathbb{R}_{+}}u\rho\left(du\right)<\infty$, $\mathbb{E}\left[e^{-\lambda U}\right]=\exp-\left[\lambda a+\int_{\mathbb{R}_{+}}1-e^{-\lambda u}\rho\left(du\right)\right]$
with $a\geq0$ and $\mathbb{E}\left[U\right]=a+\int_{\mathbb{R}_{+}}u\rho\left(du\right)$
(see \cite{Sato99LevPro}, page 285).

Now, an element of $\mathbb{C}\oplus\mathfrak{C}_{Y}$ is of the form
$M_{Y}\left(g\right)+b$ where $M_{Y}\left(g\right)$ is the centered
stochastic integral of the function $g\in L^{2}\left(\nu\right)$
and $b$ a constant. It is an ID random variable without Gaussian
part and its Lévy measure is $\rho:=\left(\nu\circ g^{-1}\right)_{|\mathbb{R}^{\star}}$.
If now $M_{Y}\left(g\right)+b$ is non-negative, we deduce that $\nu\left(g^{-1}\left(\mathbb{R}_{-}^{\star}\right)\right)=\rho\left(\mathbb{R}_{-}^{\star}\right)=0$,
thus $g\geq0$. Since by hypothesis, $M_{Y}\left(g\right)+b\in L^{2}\left(\nu^{*}\right)$,
it is also integrable thus $\int_{\mathbb{R}_{+}}u\rho\left(du\right)=\int_{Y}g\left(y\right)\mu\left(dy\right)<\infty$.
Then $g\in L^{1}\left(\nu\right)\cap L^{2}\left(\nu\right)$ and we
can deduce that $M_{Y}\left(g\right)=\int_{Y}gdN_{Y}-\int_{Y}gd\nu$.
Finally, since $\mathbb{E}\left[M_{Y}\left(g\right)+b\right]=b$,
$b\geq\int_{Y}gd\nu$ and the result follows with $f=g$ and $c=b-\int_{Y}gd\nu$.

It is enough to check the positivity of $\Theta$ on indicator functions
of sets of finite $\mu$-measure. We have, by construction, $\Phi M_{X}\left(1_{A}\right)=M_{Y}\left(\Theta1_{A}\right)$,
but since $1_{A}$ is in $L^{1}\left(\mu\right)\cap L^{2}\left(\mu\right)$,
$M_{X}\left(1_{A}\right)=\int_{X}1_{A}dN_{X}-\int_{X}1_{A}d\mu$ and
$\Phi M_{X}\left(1_{A}\right)=\Phi\int_{X}1_{A}dN_{X}-\Phi\int_{X}1_{A}d\mu$.
$\int_{X}1_{A}dN_{X}$ is a non-negative random variable, then $\Phi\int_{X}1_{A}dN_{X}$
is also non-negative. From the first part of the proof, there exists
$f\in L^{1}\left(\nu\right)\cap L^{2}\left(\nu\right)$, $f\geq0$
and $c\geq0$ such that $\Phi\int_{X}1_{A}dN_{X}=\int_{Y}fdN_{Y}+c$.

Finally, since $\mathbb{E}\left[\Phi\int_{X}1_{A}dN_{X}\right]=\int_{X}1_{A}d\mu=\int_{Y}fd\nu+c$,
$\Phi M_{X}\left(1_{A}\right)=\int_{Y}fdN_{Y}-\int_{Y}fd\nu=M_{Y}\left(f\right)$.
By identification, $\Theta1_{A}=f\geq0$ and since $c\geq0$, $\int_{Y}\Theta1_{A}d\nu\leq\int_{X}1_{A}d\mu$.

Note that it is not true in general that, if $g$ is non-negative
and in $L^{2}\left(\nu\right)$, there exists $b$ such that $M_{Y}\left(g\right)+b$
is non-negative.

By construction, the adjoint $\Theta^{\star}$ satisfies the same
properties as $\Theta$.

$\lambda\left(A\times B\right)=\int_{B}\Theta1_{A}d\nu$ defines a
measure on $\left(X\times Y,\mathcal{A}\otimes\mathcal{B}\right)$
such that its projections satisfies $\lambda\left(\cdot\times Y\right)\leq\mu$
and $\lambda\left(X\times\cdot\right)\leq\nu$. Therefore, $\Theta$
is a sub-Markov operator.
\end{proof}
As a direct consequence of the last proposition, we can now formulate
a theorem analogous to Gaussian joinings of Gaussian dynamical systems:

\begin{thm}
A joining between Poisson suspensions is Poissonian if and only the
associated Markov operator is an exponential.
\end{thm}
Note that, from a structural point of view, it is not true in general
that a Poissonian joining is itself (isomorphic to) a Poisson suspension.
We are thus led to seek criteria ensuring that the Poissonian joining
\[
\left(X^{*}\times Y^{*},\mathcal{A}^{*}\otimes\mathcal{B}^{*},\left(\mu^{\prime*}\otimes\delta_{\mu_{0}}\right)*\widetilde{\gamma^{*}}*\left(\delta_{\mu_{0}}\otimes\nu^{\prime*}\right),T_{*}\times S_{*}\right)\]
 is isomorphic to the Poisson suspension \[
\left(\left(X\times Y\right)^{*},\left(\mathcal{A}\otimes\mathcal{B}\right)^{*},\left(\mu^{\prime}\otimes\delta_{0}+\gamma+\delta_{0}\otimes\nu^{\prime}\right)^{*},\left(T\times S\right)^{*}\right)\]
 and not just a factor. This is obviously the case if $\gamma=0$
or if $\gamma$ is a graph joining or a sum of graph joinings. We
prove this result in a non-trivial case which is of particular importance.

\begin{prop}
\label{pro:Poissonjoinings=00003DPoisson}In a suspension $\left(\left(X\times Y\right)^{*},\left(\mathcal{A}\otimes\mathcal{B}\right)^{*},m^{*}\right)$,
if $\mathcal{A}\times Y\cap X\times\mathcal{B}$ is non-atomic mod.
$m$ then $\left(\mathcal{A}\times Y\right)^{*}\vee\left(X\times\mathcal{B}\right)^{*}$
equals $\left(\mathcal{A}\otimes\mathcal{B}\right)^{*}$ mod. $m^{*}$.
Equivalently, in a suspension $\left(X^{*},\mathcal{A}^{*},\mu^{*}\right)$,
$\mathcal{C}_{1}\subset\mathcal{A}$, $\mathcal{C}_{2}\subset\mathcal{A}$
: if $\mathcal{C}_{1}\cap\mathcal{C}_{2}$ is non atomic, then $\mathcal{C}_{1}^{*}\vee\mathcal{C}_{2}^{*}=\left(\mathcal{C}_{1}\vee\mathcal{C}_{2}\right)^{*}$. 
\end{prop}
\begin{proof}
The problem can be translated into the following one: we have to reconstruct
a Poisson measure on $X\times Y$ knowing its projections on $X$
and $Y$ respectively.

Note that the fact that $\mathcal{F}:=\mathcal{A}\times Y\cap X\times\mathcal{B}$
is not trivial implies that $m_{\diagup\mathcal{F}}$ is supported
on a graph, say $\Phi$ between $\left(X_{\diagup\mathcal{F}},\mathcal{A}_{\diagup\mathcal{F}},\varphi_{1*}\left(m\right)\right)$
and $\left(Y_{\diagup\mathcal{F}},\mathcal{B}_{\diagup\mathcal{F}},\varphi_{2*}\left(m\right)\right)$
, if we denote by $\varphi_{1}$ and $\varphi_{2}$ the factor maps
from $X$ to $X_{\diagup\mathcal{F}}$ and from $Y$ to $Y_{\diagup\mathcal{F}}$
respectively (with a slight abuse, we consider $\mathcal{F}\subset\mathcal{A}$
and $\mathcal{F}\subset\mathcal{B}$ ).

The measure $m$ gives full measure to the set $A$ of points $\left(x,y\right)$
such that $\Phi\circ\varphi_{1}\left(x\right)=\varphi_{2}\left(y\right)$,
therefore $m^{*}$ gives full measure to the set of the points $\nu=\sum_{k\in\mathbb{N}}\delta_{\left(x_{k},y_{k}\right)}\in\left(X\times Y\right)^{*}$
such that, for any $k$, $\left(x_{k},y_{k}\right)\in A$. Moreover,
if $k\neq k^{\prime}$, then $\varphi_{1}\left(x_{k}\right)\neq\varphi_{1}\left(x_{k^{\prime}}\right)$
because if it were not the case, the Poisson measure $\left(\left(X_{\diagup\mathcal{F}}\right)^{*},\left(\mathcal{A}_{\diagup\mathcal{F}}\right)^{*},\left(\varphi_{1*}\left(m\right)\right)^{*}\right)$
would have multiplicities and this is impossible since $\varphi_{1*}\left(m\right)$
is a continuous measure, as $\mathcal{F}$ was assumed non atomic.

Thus, from the projections $\sum_{i\in\mathbb{N}}\delta_{x_{i}}$
and $\sum_{j\in\mathbb{N}}\delta_{y_{j}}$ coming from $\nu=\sum_{k\in\mathbb{N}}\delta_{\left(x_{k},y_{k}\right)}$,
we know how to reconstruct $\nu$ (and this is the only way to do
it) since to every point $x_{i}$ there is exactly one point $x_{j}$
such that $\left(x_{i},y_{j}\right)\in A$ from the preceding discussion.
At last, as we've already said, each pair of point $\left(x_{k},y_{k}\right)$
of $\nu$ satisfy $\left(x_{k},y_{k}\right)\in A$, thus there is
no ambiguity in the reconstruction of $\nu$ from its projections.
This shows that $\left(\mathcal{A}\times Y\right)^{*}\vee\left(X\times\mathcal{B}\right)^{*}$
is nothing else than the whole $\sigma$-algebra $\left(\mathcal{A}\otimes\mathcal{B}\right)^{*}$.
\end{proof}
We can now treat the important case of relatively independent joinings:

\begin{prop}
\label{pro:relweakmixing}Consider two dynamical systems $\left(X,\mathcal{A},\mu,T\right)$
and $\left(Y,\mathcal{B},\nu,S\right)$, with a common $\sigma$-finite
factor $\left(Z,\mathcal{C},m,R\right)$, and consider the relatively
independent joining $\left(X\times Y,\mathcal{A}\otimes\mathcal{B},\mu\otimes_{\mathcal{C}}\nu,T\times S\right)$
over $\left(Z,\mathcal{C},m,R\right)$. Then $\left(X^{*}\times Y^{*},\mathcal{A}^{*}\otimes\mathcal{B}^{*},\widetilde{\left(\mu\otimes_{\mathcal{C}}\nu\right)^{*}},T_{*}\times S_{*}\right)$
is the relatively independent joining of $\left(X^{*},\mathcal{A}^{*},\mu^{*},T_{*}\right)$
and $\left(Y^{*},\mathcal{B}^{*},\nu^{*},S_{*}\right)$ over $\left(Z^{*},\mathcal{C}^{*},m^{*},R_{*}\right)$,
i.e. $\widetilde{\left(\mu\otimes_{\mathcal{C}}\nu\right)^{*}}=\mu^{*}\otimes_{\mathcal{C}^{*}}\nu^{*}$.
Moreover, if $m$ is a continuous measure, it is isomorphic to $\left(\left(X\times Y\right)^{*},\left(\mathcal{A}\otimes\mathcal{B}\right)^{*},\left(\mu\otimes_{\mathcal{C}}\nu\right)^{*},\left(T\times S\right)^{*}\right)$
(this condition is satisfied for example if $T$ is $\mathbf{II}_{\infty}$).

If a Poissonian joining is a relatively independent joining over a
factor $\mathcal{F}$, then $\mathcal{F}$ is a Poissonian factor.
\end{prop}
\begin{proof}
The Markov operator that describes a relatively independent joining
over a $\sigma$-finite factor is the conditional expectation operator
on this factor (see \cite{Glas03Ergojoin} page 129). It is now sufficient
to recall that the exponential of the conditional expectation of a
factor on the base is the conditional expectation of the corresponding
Poissonian factor (see Section \ref{sub:Poissonian-map-and}). The
second point is a direct application of Proposition \ref{pro:Poissonjoinings=00003DPoisson}.

For the last result, consider a Poissonian joining which is a relatively
independent joining over a factor $\mathcal{F}$. Given the structure
result of the distribution of any Poissonian joining (Theorem \ref{thm:structureIDjoinings}),
we can assume that $\mu^{\prime}$ and $\nu^{\prime}$ are zero. Now,
we only have to remark that a joining (in finite or infinite measure)
is a relatively independent joining if and only if the associated
Markov operator is an orthogonal projection. The Markov operator $E_{\mathcal{F}}$
associated to the Poissonian joining is the exponential of an operator
$p$ associated to the joining on the base. But $E_{\mathcal{F}}$
is an orthogonal projection and preserves the first chaos, thus its
restriction to it, which is identified to $p$ is also an orthogonal
projection and a Markov operator as well. Thus, there exists a factor
$\mathcal{B}$ such that $p=E_{\mathcal{B}}$, but as we have already
seen, the exponential of a conditional expectation is the conditional
expectation on the corresponding Poissonian factor, and we deduce
$\mathcal{F}=\mathcal{B}^{*}$, i.e., $\mathcal{F}$ is a Poissonian
factor.
\end{proof}
\begin{rem*}
It can be shown that Poissonian selfjoinings of weakly mixing (=ergodic,
here) (resp. mildly mixing, resp. mixing, resp. $K$) suspensions
are also weakly mixing (resp. mildly mixing, resp. mixing, resp. $K$).
From the last proposition, we can conclude that a weakly mixing Poisson
suspension is always relatively weakly mixing over its Poissonian
factors.
\end{rem*}

\subsection{An example}

Let us give a concrete example on a particularly simple system. Aaronson
and Nadkarni in \cite{Aar87Eigen} have proved the existence of a
system $\left(X,\mathcal{A},\mu,T\right)$ (said to have \emph{minimal
self joinings}) whose only ergodic self-joinings are of the form $\mu\left(A\cap T^{-k}B\right)$.
It is easy to describe all Poissonian selfjoinings of the corresponding
Poisson suspension. They are of the form:

\[
\mathfrak{P}=\left(\left(a\mu\right)^{*}\otimes\left(a\mu\right)^{*}\right)*\dots*\left(c_{-k}\nu_{-k}\right)^{*}*\dots*\left(c_{0}\nu_{0}\right)^{*}*\dots*\left(c_{k}\nu_{k}\right)^{*}*\dots\]

where $\nu_{k}\left(A\times B\right)=\mu\left(A\cap T^{-k}B\right)$
and ${\displaystyle \sum_{k\in\mathbb{Z}}}c_{k}=1-a$ with $a\geq0$,
$c_{k}\geq0$, $k\in\mathbb{Z}$.

In this case, Poissonian joinings are reduced to the minimum possible.
In this example, the Poissonian centralizer (i.e. Poissonian automorphisms
commuting with $T_{*}$) is trivial; that is, it consists only in
the powers of $T_{*}$.

This example shows that an ergodic Poisson suspension (which always
has the aforementioned selfjoinings) is never simple, indeed, by taking
$c_{1}>0$ and $c_{2}>0$, the corresponding Poissonian joining is
ergodic and not a graph joining nor the product joining.

\subsection{Strong disjointness and ID-disjointness.}

We recall (see \cite{Aar97InfErg}) that a $\left(c_{1},c_{2}\right)$-joining
between $\left(X,\mathcal{A},\mu,T\right)$ and $\left(Y,\mathcal{B},\nu,S\right)$
is a joining between $\left(X,\mathcal{A},c_{1}\mu,T\right)$ and
$\left(Y,\mathcal{B},c_{2}\nu,S\right)$. Two systems are strongly
disjoint if there is no $\left(c_{1},c_{2}\right)$-joining possible.

We have the following links between strong disjointness and ID-disjointness.

\begin{thm}
\label{thm:disfortdisID}Two dynamical systems are strongly disjoint
if their Poisson suspensions are ID-disjoint.

Two ergodic dynamical systems are strongly disjoint if and only if
their associated Poisson suspensions are ID-disjoint.
\end{thm}
\begin{proof}
Assume there exists a $\left(c_{1},c_{2}\right)$-joining $\gamma$
between $\left(X,\mathcal{A},\mu,T\right)$ and $\left(Y,\mathcal{B},\nu,S\right)$.
We can assume that $c_{1}\geq c_{2}$ and form the non trivial Poissonian
joining with distribution $\widetilde{\left(\frac{1}{c_{1}}\gamma\right)^{*}}*\left(\delta_{\mu_{0}}\otimes\left(\left(1-\frac{c_{2}}{c_{1}}\right)\nu\right)^{*}\right)$.

If the base systems are ergodic and their suspensions not ID-disjoint
there exists a non trivial Poissonian joining of the form $\left(\mu^{\prime*}\otimes\delta_{\mu_{0}}\right)*\widetilde{\gamma^{*}}*\left(\delta_{\mu_{0}}\otimes\nu^{\prime*}\right)$.
By ergodicity, there exists $a_{1}$ and $a_{2}$ such that $\mu^{\prime}=a_{1}\mu$
and $\nu^{\prime}=a_{2}\nu$, therefore $\gamma$ is a $\left(\left(1-a_{1}\right),\left(1-a_{2}\right)\right)$-joining
between $\left(X,\mathcal{A},\mu,T\right)$ and $\left(Y,\mathcal{B},\nu,S\right)$.
\end{proof}
Since disjointness implies ID-disjointness, a useful immediate corollary
is the following:

\begin{cor}
\label{cor:disjointimpliesstrongdis}Two dynamical systems are strongly
disjoint if their Poisson suspensions are disjoint.
\end{cor}
This corollary is a powerful tool to prove strong disjointness results.
Let us illustrate it by an example. In \cite{Jan07entropy}, Poisson
entropy of a transformation $T$ is defined as the Kolmogorov entropy
of its Poisson suspension. Poisson entropy is shown to coincide, for
quasi-finite systems (see \cite{Kren67Entropy} for the definition),
with Krengel's and Parry's definition of entropy (see \cite{Par69Entropy}),
moreover it is proved a dichotomy for these systems: $T$, quasi-finite
and ergodic, is either a remotely infinite system, or it possesses
a Pinsker ($\sigma$-finite)-factor. It is conjectured that this dichotomy
holds for any system; quasi-finite systems constitutes a very large
class anyway.

We now prove a strong disjointness result for a large class of systems
which includes quasi-finite ones in particular. By Theorem \ref{thm:ergopropPoissusp},
if $T$ is remotely infinite then $T_{*}$ is $K$ and by a well known
result, in finite measure, $K$-systems are disjoint from zero entropy
ones. Now, using the fact that to be not strongly disjoint is an equivalence
relation, Corollary \ref{cor:disjointimpliesstrongdis} implies:

\begin{thm}
If $T$ is a remotely infinite system and if $S$ has a zero Poisson
entropy ($\sigma$-finite) factor, then they are strongly disjoint.
\end{thm}

\subsection{ID-disjointness and spectral theory}

The definition of ID-disjointness can be extended to other infinitely
divisible systems and notably stationary ID processes as studied in
\cite{Maruyama70IDproc}. A Poisson suspension $\left(X^{*},\mathcal{A}^{*},\mu^{*},T_{*}\right)$
and a stationary ID process $\left(\mathbb{R}^{\mathbb{Z}},\mathcal{B}^{\otimes\mathbb{Z}},\mathbb{P},S\right)$
are ID-disjoint if the only way to join them with an ID distribution
(on $X^{*}\times\mathbb{R}^{\mathbb{Z}}$) is the product measure.

\begin{prop}
\label{pro:IDdisjPoissonGauss}A Poisson suspension and a stationary
Gaussian processes are ID-disjoint.
\end{prop}
\begin{proof}
Assume there exists a joining between a Poisson suspension $N$ and
a stationary Gaussian process $\left\{ X_{n}\right\} _{n\in\mathbb{Z}}$
such that the pair $\left(N,\left\{ X_{n}\right\} _{n\in\mathbb{Z}}\right)$
has an infinitely divisible distribution on $X^{*}\times\mathbb{R}^{\mathbb{Z}}$.
It is clear that, for all finite family $A_{1},\dots,A_{k}\in\mathcal{A}$,
$0<\mu\left(A_{i}\right)<\infty$ , and all finite subset $\left\{ i_{1},\dots,i_{l}\right\} $
of $\mathbb{Z}$, the random vector (in $\mathbb{R}^{k+l}$) $\left\{ N\left(A_{1}\right),\dots,N\left(A_{k}\right),X_{1},\dots,X_{l}\right\} $
is ID. From the Lévy-Khintchine formula, such a vector is always the
independent sum $\left\{ Y_{1},\dots,Y_{k+l}\right\} +\left\{ Z_{1},\dots,Z_{k+l}\right\} $
of an ID random vector without Gaussian part and a Gaussian random
vector. The non-negativity of $N\left(A_{i}\right),i\leq k$ implies
that $Z_{i}=0$, $i\leq k$ and Cramér Theorem implies that $Y_{i}$,
$i>k$ is Gaussian and hence zero. This implies that $\left\{ N\left(A_{1}\right),\dots,N\left(A_{k}\right)\right\} $
and $\left\{ X_{1},\dots,X_{l}\right\} $ are independent. It follows
that $N$ and $\left\{ X_{n}\right\} _{n\in\mathbb{Z}}$ are independent,
thus proving ID-disjointness.
\end{proof}
Surprisingly, ID-disjointness with Gaussian systems will serve to
prove a general result on spectral theory of infinite measure preserving
systems.

Gaussian systems are spectral in nature. They are fully described
by a symmetric measure on the circle and it may happen that some measures
on the circle, appearing as spectral measures of a dynamical system
forces the Gaussian character of the associated stationary process.
These measures are said to possess the {}``Foïa\c{$\mathrm{s}$}
and Str$\breve{\mathrm{a}}$til$\breve{\mathrm{a}}$'' property.

\begin{defn}
(See \cite{LemParThou00Gausselfjoin}) A symmetric measure $\sigma$
on $\left[-\pi,\pi\right[$, possesses the {}``Foïa\c{$\mathrm{s}$}
and Str$\breve{\mathrm{a}}$til$\breve{\mathrm{a}}$'' $\left(FS\right)$
property if, for any ergodic probability preserving dynamical system
$\left(Y,\mathcal{C},\nu,S\right)$, every $f\in L_{0}^{2}\left(\nu,S\right)$
with spectral measure $\sigma$, is a Gaussian random variable.
\end{defn}
Gaussian processes associated to these measures have led to deep and
beautiful results on their structure (see \cite{LemParThou00Gausselfjoin}).

\begin{thm}
\label{thm:Foias}A measure with the $\left(FS\right)$ property is
never the spectral measure of a vector in a $\mathbf{II}_{\infty}$
system.
\end{thm}
\begin{proof}
Assume there exists a $\mathbf{II}_{\infty}$ system $\left(X,\mathcal{A},\mu,T\right)$
and a vector $f$ such that $\sigma_{f}$ has the $\left(FS\right)$
property. Consider now $\left(M_{X},\mathcal{M}_{\mathcal{A}},\mathcal{P}_{\mu},\widetilde{T}\right)$,
it is ergodic by Theorem \ref{thm:ergopropPoissusp}. The centered
stochastic integral $M\left(f\right)\in\mathfrak{C}$ has spectral
measure $\sigma_{f}$. The stationary process $\left\{ M\left(f\right)\circ\widetilde{T}^{n}\right\} _{n\in\mathbb{Z}}$
is ergodic and then Gaussian since $\sigma_{f}$ has the $\left(FS\right)$
property. But from Lemma \ref{pro:IDdisjPoissonGauss}, $M\left(f\right)$
is necessarily zero.
\end{proof}
The following corollary can be seen as a slight extension of the Foïa\c{$\mathrm{s}$}
and Str$\breve{\mathrm{a}}$til$\breve{\mathrm{a}}$ property.

\begin{cor}
Let $\left(\Omega,\mathcal{F},\mu,T\right)$ be a conservative dynamical
system and $f\in L^{2}\left(\mu\right)$ with spectral measure $\sigma$
that possesses the $\left(FS\right)$ property, then $f=f1_{\mathfrak{P}}$
mod. $\mu$ where $\mathfrak{P}$ is the part of $\left(\Omega,\mathcal{F},\mu,T\right)$
of type $\mathbf{II}_{1}$.
\end{cor}
\begin{proof}
Let $\mu_{\infty}:=\mu_{|\mathfrak{P}^{c}}$ and $\mu_{f}:=\mu_{|\mathfrak{P}}$.
We have:

$\hat{\sigma}\left(k\right)={\displaystyle \int_{\Omega}}ff\circ T^{k}d\mu$

\[
={\displaystyle \int_{\Omega}}\left(f1_{\mathfrak{P}}\right)\left(f1_{\mathfrak{P}}\right)\circ T^{k}d\mu_{f}+{\displaystyle \int_{\Omega}}\left(f1_{\mathfrak{P^{c}}}\right)\left(f1_{\mathfrak{P^{c}}}\right)\circ T^{k}d\mu_{\infty}=\hat{\sigma}_{f}\left(k\right)+\hat{\sigma}_{\infty}\left(k\right)\]
where $\sigma_{f}+\sigma_{\infty}=\sigma$.

This yields to the relation of absolute continuity $\sigma_{\infty}\ll\sigma$
and as such, $\sigma_{\infty}$ possesses the $\left(FS\right)$ property
too (see \cite{LemParThou00Gausselfjoin}) unless it is the zero measure.
But the system $\left(\Omega,\mathcal{F},\mu_{\infty},T\right)$,
of type $\mathbf{II}_{\infty}$, would have a vector, $f1_{\mathfrak{P^{c}}}$,
with a spectral measure that possesses the $\left(FS\right)$ property.
This is not possible from Theorem \ref{thm:Foias}.  Hence $\sigma_{\infty}$
is zero and so is $f1_{\mathfrak{P^{c}}}$, that is $f=f1_{\mathfrak{P}}$
mod. $\mu$.
\end{proof}

\section{{}``Simple'' Poisson suspensions}

At first sight, a Poisson suspension seems to be a much more complex
object than the base on which it is built. The transformation $T$
sends one point to another while $T_{*}$ maps a countable set of
points to another. The aim of this chapter is to focus on Poisson
suspensions were the gap of complexity is considerably reduced.

First, spectrally, the complexity will be equivalent if both the suspension
and the base have simple spectrum. We will see that simple spectrum
(indeed, slightly less restrictive hypothesis will suffice) will force
the centralizer of both transformation to be the same and force the
natural factors of a Poisson suspension we can think of to be Poissonian
factors. Moreover, the particularly clear structure will enable us
to compare this systems to the apparently similar looking Gaussian
systems.

It is easily deduced that a Poisson suspension will have simple spectrum
if and only if it has simple spectrum on each chaos and if, for all
$m\neq n$, $\sigma^{*m}\perp\sigma^{*n}$. We thus are led to seek
criteria under which these conditions are satisfied.

Ageev in \cite{Ageev99homspec} (see also \cite{Ageev00cartpower})
has proved the following result (we present a reduced version of it):

\begin{thm}
\label{thm:Ageev}Let $U$ be a unitary operator acting on a complex
Hilbert space $H$. Assume $U$ has simple and continuous spectrum
and moreover, that, for countably many $\theta$, $0<\theta<1$, $\theta Id+\left(1-\theta\right)U$
belongs to the weak closure of powers of $U$.

Then, for all $n\geq1$ $U^{\otimes n}$ has simple spectrum on $\mathfrak{S}^{n}H$
(the $n$-th symmetric tensor power of $H$) and the corresponding
maximal spectral types are all pairwise mutually singular.
\end{thm}
Therefore, if $T$ is such that $U_{T}$ satisfies assumptions of
Theorem \ref{thm:Ageev} then $T_{*}$ has simple spectrum.

It is easy to construct an appropriate rank one transformation $T$
in infinite measure such that $U_{T}$ satisfies assumptions of Theorem
\ref{thm:Ageev}, the existence of a simple spectrum suspension is
thus ensured.

In the sequel, we will often consider a base $\left(X,\mathcal{A},\mu,T\right)$
whose maximal spectral type $\sigma$ satisfies

\begin{equation}
\sigma\perp{\displaystyle \sum_{k=2}^{+\infty}}\frac{1}{k!}\sigma^{*k}\label{eq:singularconvolution}\end{equation}

and will refer to it as satisfying property (\ref{eq:singularconvolution}).

It concerns of course Poisson suspension with simple spectrum.

\subsection{Centralizer}

We begin with a result similar to the Gaussian case:

\begin{prop}
Let $T$ and $S$ with property (\ref{eq:singularconvolution}). Assume
there exists an automorphism $\varphi$ such that $T_{*}\varphi=\varphi S_{*}$,
then $\varphi$ is indeed a Poissonian automorphism . In particular
$C\left(T_{*}\right)\simeq C\left(T\right)$.
\end{prop}
\begin{proof}
Take $f$ in the first chaos. $f\circ\varphi$ has the same spectral
measure under $S_{*}$ as $f$ under $T_{*}$, thus, since the latter
is absolutely continuous with respect to $\sigma$ and that $\sigma\perp{\displaystyle \sum_{k=2}^{+\infty}}\frac{1}{k!}\sigma^{*k}$,
$f\circ\varphi$ also belongs to the first chaos. This proves that
$\varphi$, seen as a Markov operator on $L^{2}\left(\mu^{*}\right)$,
preserves the first chaos; therefore, thanks to Proposition \ref{pro:Markov}
$\varphi$ induces a sub-Markov operator $\Phi$ on $L^{2}\left(\mu\right)$
that intertwines $T$ and $S$. But $\varphi$ is also an isometry
thus $\Phi$ is an isometry of $L^{2}\left(\mu\right)$ and this implies,
since $\Phi$ is also a sub-Markov operator, that $\Phi$ is indeed
a Markov operator which comes from an automorphism, say $R$. We have
thus proved that, for each, $A\in\mathcal{A}$, $0<\mu\left(A\right)<\infty$:

\[
N\left(A\right)\circ\varphi=N\left(R^{-1}A\right)\]

Under the graph measure on $X^{*}\times X^{*}$ given by $\mu^{*}\left(A\cap\varphi^{-1}B\right)$,
the bivariate Poisson measure takes the form $\left(N\left(\cdot\right),N\left(\cdot\right)\circ\varphi\right)=\left(N\left(\cdot\right),N\left(R^{-1}\left(\cdot\right)\right)\right)$
and this proves that this graph measure is ID. The graph joining given
by $\varphi$ between $T_{*}$ and $S_{*}$ is therefore Poissonian
and we can deduce that $\varphi$, as a Markov operator, is the exponential
of $\Phi$ which is in turn corresponds to $R^{*}$, that is $\varphi=R^{*}$.
\end{proof}

\subsection{Comparison with Gaussian systems}

So far, it was not clear that Poisson suspensions are not (isomorphic
to) Gaussian systems. The next two propositions will show, hopefully,
that we are not reasoning on the same objects. The first one has been
proved by François Parreau (unpublished).

\begin{prop}
A suspension $\left(X^{*},\mathcal{A}^{*},\mu^{*},T_{*}\right)$ with
$T$ satisfying property (\ref{eq:singularconvolution}) is not isomorphic
to a Gaussian system.
\end{prop}
\begin{proof}
Suppose there exists a Gaussian system $\left(Y,\mathcal{Y},\nu,S\right)$
isomorphic to $\left(X^{*},\mathcal{A}^{*},\mu^{*},T_{*}\right)$
by an isomorphism $\varphi$. Call $m$ the maximal spectral type
of the first chaos of $L^{2}\left(\nu\right)$. Due to the assumptions
on the spectral structure of $L^{2}\left(\mathcal{P}_{\mu}\right)$,
we can't have $m\perp\sigma$. So, consider $\rho$ such that $\rho\ll m$
and $\rho\ll\sigma$. Take $f$, a vector in the first chaos of $L^{2}\left(\nu\right)$
with spectral measure $\rho$, $f$ is thus a Gaussian random variable.
Now the vector $f\circ\varphi^{-1}$ belongs to $L^{2}\left(\mu^{*}\right)$,
has spectral measure $\rho$ and $\left\{ f\circ\varphi^{-1}\circ T_{*}^{n}\right\} _{n\in\mathbb{Z}}$
is a Gaussian stationary process under $\mu^{*}$. But since $\rho\ll\sigma$,
we have $\rho\perp{\displaystyle \sum_{k=2}^{+\infty}}\frac{1}{k!}\sigma^{*k}$
which implies that, for all $n\in\mathbb{Z}$, $f\circ\varphi^{-1}\circ T_{*}^{n}$
belongs to the first chaos of $L^{2}\left(\mu^{*}\right)$ which implies
that the graph joining is ID which is impossible by ID-disjointness
(Proposition \ref{pro:IDdisjPoissonGauss}).
\end{proof}
It is well known that Gaussian systems are always isomorphic to an
arbitrary number of direct products of transformations. We will show
in the next proposition, that it is no longer the case for some Poisson
suspensions. This establishes a drastic difference with Gaussian systems.

\begin{prop}
Let $T$ with property (\ref{eq:singularconvolution}). $\left(X,\mathcal{A},\mu,T\right)$
is ergodic if and only if there don't exist two non-trivial $T_{*}$-invariant
and independent sub-$\sigma$-algebra $\mathcal{B}$ and $\mathcal{C}$
such that $\mathcal{A}^{*}=\mathcal{B}\vee\mathcal{C}$, i.e. $\left(X^{*},\mathcal{A}^{*},\mu^{*},T_{*}\right)$
is not isomorphic to a non trivial direct product.
\end{prop}
\begin{proof}
Assume $\left(X,\mathcal{A},\mu,T\right)$ ergodic. Let $\mathcal{B}$
be a non trivial $T_{*}$-invariant sub-$\sigma$-algebra of $\mathcal{A}^{*}$
and suppose there exists another $T_{*}$-invariant sub-$\sigma$-algebra
$\mathcal{C}$ independent of $\mathcal{B}$ and such that $\mathcal{A}^{*}=\mathcal{B}\vee\mathcal{C}$.
Denote by $\nu_{a}$ (resp. $\nu_{b}$) the (reduced) maximal spectral
type of the quotient action $T_{*\diagup\mathcal{B}}$ (resp. $T_{*\diagup\mathcal{C}}$).
We must have, up to equivalence of measures, ${\displaystyle \sum_{k=1}^{+\infty}}\frac{1}{k!}\sigma^{*k}=\nu_{a}+\nu_{b}+\nu_{a}*\nu_{b}$.
Assume $\sigma\perp\nu_{b}$, we have thus $\nu_{b}\ll{\displaystyle \sum_{k=2}^{+\infty}}\frac{1}{k!}\sigma^{*k}$
which implies that $\nu_{a}*\nu_{b}\ll{\displaystyle \sum_{k=2}^{+\infty}}\frac{1}{k!}\sigma^{*k}$
and then $\sigma\ll\nu_{a}$ since $\sigma\perp{\displaystyle \sum_{k=2}^{+\infty}}\frac{1}{k!}\sigma^{*k}$.
This means that all vectors with spectral measure $\sigma$ are measurable
with respect to $\mathcal{B}$. But these vectors are exactly those
of the first chaos and this vectors generate the whole $\sigma$-algebra
$\mathcal{A}^{*}$, that is $\mathcal{A}^{*}=\mathcal{B}$ which is
a contradiction with the initial assumptions. By symmetry, we deduce
that $\sigma\not\perp\nu_{a}$ and $\sigma\not\perp\nu_{b}$, and
this implies the existence of two centered stochastic integrals in
the first chaos $M\left(f_{a}\right)$ and $M\left(f_{b}\right)$
measurable with respect to $\mathcal{B}$ and $\mathcal{C}$ respectively.
Then  $\left\{ M\left(f_{a}\right)\circ T_{*}^{n}\right\} _{n\in\mathbb{Z}}$
and $\left\{ M\left(f_{b}\right)\circ T_{*}^{n}\right\} _{n\in\mathbb{Z}}$
are independent which is in fact impossible because, by using a result
of Maruyama (see \cite{Maruyama70IDproc}), this would imply the disjointness
of the support of the functions $\left\{ f_{a}\circ T^{n}\right\} _{n\in\mathbb{Z}}$
with the support of the functions $\left\{ f_{b}\circ T^{n}\right\} _{n\in\mathbb{Z}}$
and these two non trivial $T$-invariant sets cannot be disjoint since
$\left(X,\mathcal{A},\mu,T\right)$ is ergodic.

The converse is obvious.
\end{proof}

\subsection{Poissonian factors and IDp processes}

We will study factors of a Poisson suspension generated by processes
which are stochastic integrals. Our aim is to show that they have
the expected form, that is, they are Poissonian factors. We will need
to combine Proposition \ref{pro:Markov} with the next lemma.

\begin{lem}
\label{lem:Poisfactor}Consider a factor $\mathcal{B}$ of $\left(X^{*},\mathcal{A}^{*},\mu^{*},T_{*}\right)$
and the corresponding conditional expectation $E_{\mathcal{B}}$.
Assume $E_{\mathcal{B}}$ preserves the first chaos $\mathfrak{C}$
and consider $\mathcal{C}$, the factor of $\left(X,\mathcal{A},\mu,T\right)$
generated by those $g\in L^{2}\left(\mu\right)$ such that $I\left(g\right)\in E_{\mathcal{B}}\left(\mathfrak{C}\right)$.
Then $\mathcal{C}^{*}\subset\mathcal{B}$.
\end{lem}
\begin{proof}
We denote by $V$ the sub-Markov operator (thanks to Proposition \ref{pro:Markov})
induced by $E_{\mathcal{B}}$ on $L^{2}\left(\mu\right)$, $V$ is
an orthogonal projector. From $\int_{X}V1_{A}d\mu\leq\mu\left(A\right)$,
the $T$-invariant measure $A\mapsto\int_{X}V1_{A}d\mu$ is absolutely
continuous with respect to $\mu$. Thus we have $\int_{X}V1_{A}d\mu=\int_{X}1_{A}pd\mu$
for a function $p\leq1$. Take $\left\{ f_{n}\right\} _{n\in\mathbb{N}}$
a linearly dense family of positive functions, and form $f=\sum_{n\in\mathbb{N}}a_{n}f_{n}$
such that, $a_{n}>0$ and $f\in L^{2}\left(\mu\right)$. Let $S$
be the support of $Vf$ (it is a $T$-invariant set), since $V$ is
a projector:

\[
\int_{X}Vfd\mu=\int_{X}V\left(Vf\right)d\mu=\int_{X}Vfpd\mu\]

Thus $p\equiv1$ on $S$. It is immediate to see that $Vg=0$ if $g$
is zero on $S$. Take $g>0$, as $gVf=0$, we have $\left\langle Vg,f\right\rangle =\left\langle g,Vf\right\rangle =\int_{X}gVfd\mu=0$
and, by positivity, we deduce that $\left\langle Vg,f_{n}\right\rangle =0$
for any $n\in\mathbb{N}$, thus $Vg=0$.

$V$ restricted to $L^{2}\left(\mu_{|S}\right)$ is now a Markov operator
and since it is also an orthogonal projection, it is a conditional
expectation on a $\sigma$-finite factor $\widetilde{\mathcal{C}}$
of $\left(S,\mathcal{A}_{|S},\mu_{|S},T_{|S}\right)$ and then induces
the desired (non $\sigma$-finite if $S\neq X$) factor $\mathcal{C}$
(i.e. $\mathcal{C}=\sigma\left(\widetilde{\mathcal{C}}\right)$) of
$\left(X,\mathcal{A},\mu,T\right)$. The only non zero square integrable
function measurable with respect to $\mathcal{C}$ are supported on
$S$ and their restriction to $S$ are exactly the functions measurable
with respect to $\widetilde{\mathcal{C}}$. Thus $\textrm{Im}V=L^{2}\left(X,\mathcal{C},\mu,T\right)$
and it implies that $\mathcal{C}^{*}\subset\mathcal{B}$.
\end{proof}
The assumption that $\Phi\left(\mathfrak{C}\right)\subset\mathfrak{C}$
is automatically satisfied if $T$ satisfies property (\ref{eq:singularconvolution})
since it has already been noticed in \cite{LemParThou00Gausselfjoin}
that, for spectral measures, we have $\sigma_{\Phi f}\ll\sigma_{f}$.

Our main application of the preceding lemma concerns \emph{infinitely
divisible stationary processes without Gaussian part} (abr. IDp) (this
includes, in particular all $\alpha$-stable stationary processes).
We recall that these processes are exactly those that can be obtained
as stochastic integrals with respect to a Poisson measure (see \cite{Maruyama70IDproc}
and \cite{Roy06IDstat}).

\begin{thm}
Assume that $\left(X,\mathcal{A},\mu,T\right)$ has simple spectrum,
satisfying property (\ref{eq:singularconvolution}). Let $I\left(f\right)$
be a stochastic integral of $\left(X^{*},\mathcal{A}^{*},\mu^{*},T_{*}\right)$.
Then, the factor $\mathcal{B}$ generated by the process $\left\{ I\left(f\right)\circ T_{*}^{n}\right\} _{n\in\mathbb{Z}}$
is indeed the Poissonian factor $\mathcal{C}^{*}$, where $\mathcal{C}$
is the factor of the base generated by $\left\{ f\circ T^{n}\right\} _{n\in\mathbb{Z}}$.
If $f\in L^{2}\left(\mu\right)$, then we can remove the hypothesis
on the simplicity of the spectrum.
\end{thm}
\begin{proof}
We consider vectors of the kind $\exp i\left\langle a,I\left(f\right)\right\rangle -\mathbb{E}\left[\exp i\left\langle a,I\left(f\right)\right\rangle \right]$
where $\left\langle a,I\left(f\right)\right\rangle $ is some finite
linear combination of the $I\left(f\right)\circ T_{*}^{n}$. In \cite{Roy06IDstat},
we have evaluated the spectral measure of such vectors, it is the
measure $\left|\mathbb{E}\left[\exp i\left\langle a,I\left(f\right)\right\rangle \right]\right|^{2}{\displaystyle \sum_{k=1}^{+\infty}}\frac{1}{k!}\sigma_{a}^{*k}$,
where $\sigma_{a}$ is the spectral measure of $\exp i\left\langle a,f\right\rangle -1$
in $\left(X,\mathcal{A},\mu,T\right)$.

Call $C_{a}$ the cyclic space associated to $\exp i\left\langle a,I\left(f\right)\right\rangle -\mathbb{E}\left[\exp i\left\langle a,I\left(f\right)\right\rangle \right]$,
since $\sigma_{a}\ll{\displaystyle \sum_{k=1}^{+\infty}}\frac{1}{k!}\sigma_{a}^{*k}$,
this cyclic space contains vectors with $\sigma_{a}$ as spectral
measure. Thanks to the mutual singularity of $\sigma^{*n}$ with $\sigma$,
these vectors necessarily belong to the first chaos, moreover, since
$\left(X,\mathcal{A},\mu,T\right)$ has simple spectrum, $I\left(\exp i\left\langle a,f\right\rangle -1\right)$
and its iterates stand among them. We have shown that $I\left(\exp i\left\langle a,f\right\rangle -1\right)$
are measurable with respect to $\mathcal{B}$. The factor generated
by $\exp i\left\langle a,f\right\rangle -1$, for any $a$ is exactly
$\mathcal{C}$. Thanks to Lemma \ref{lem:Poisfactor}, we have $\mathcal{C}^{*}\subset\mathcal{B}$
and the result follows as we also have the other inclusion.

The last statement is a direct application of Lemma \ref{lem:Poisfactor}.
\end{proof}
This theorem proves that many stationary IDp processes are isomorphic
to the Poisson suspension associated to their Lévy measure and not
just a factor. In particular, it reduces the study of a large class
of stationary IDp process, which may have very bad integrability properties,
to that, much more tractable, of a Poisson suspension. For example,
it solves, for this class, the following open question:

Is the maximal spectral type of any stationary IDp process of the
form ${\displaystyle \sum_{k=1}^{+\infty}}\frac{1}{k!}\sigma^{*k}$
?

It is thus maybe worthwhile to give an autonomous version of the last
theorem for IDp processes.

\begin{thm}
Let $\left\{ X_{n}\right\} _{n\in\mathbb{Z}}$ be a stationary IDp
process with Lévy measure $Q$. Assume that $\left(\mathbb{R}^{\mathbb{Z}},\mathcal{B}^{\otimes\mathbb{Z}},Q,S\right)$
has simple spectrum and $S$ has property (\ref{eq:singularconvolution}).
Then $\left\{ X_{n}\right\} _{n\in\mathbb{Z}}$ is isomorphic to $\left(\left(\mathbb{R}^{\mathbb{Z}}\right)^{*},\left(\mathcal{B}^{\otimes\mathbb{Z}}\right)^{*},Q^{*},S_{*}\right)$.
If $X_{0}$ is square integrable, then we can remove the hypothesis
on the simplicity of the spectrum.
\end{thm}
Let us finish by giving one last application of Lemma \ref{lem:Poisfactor}
that shows that Poissonian factors behave well under classical operations.
We already have pointed out (see the discussion preceding Proposition
\ref{pro:Poissonjoinings=00003DPoisson}) that, in general, we don't
have $\mathcal{C}_{1}^{*}\vee\mathcal{C}_{2}^{*}=\left(\mathcal{C}_{1}\vee\mathcal{C}_{2}\right)^{*}$,
however, by Lemma \ref{lem:Poisfactor}, this is easy to see that
the equality holds if $T$ has property (\ref{eq:singularconvolution}).

\begin{acknowledgement*}
The author would like to thank the anonymous referee for useful comments
and Mariusz Lema$\acute{\mathrm{n}}$czyk, François Parreau and Jean-Paul
Thouvenot for fruitful discussions.
\end{acknowledgement*}
\bibliographystyle{plain}
\bibliography{biblio.bib}

\end{document}